\documentclass[11pt,a4paper]{amsart}
\usepackage{amssymb}

\theoremstyle{plain}
\newtheorem{theorem}{Theorem}[section]
\newtheorem{lemma}[theorem]{Lemma}
\newtheorem{conjecture}[theorem]{Conjecture}
\theoremstyle{definition}
\newtheorem{definition}[theorem]{Definition}
\theoremstyle{remark}
\newtheorem{remark}[theorem]{Remark}
\newtheorem{example}[theorem]{Example}

\DeclareMathOperator{\ct}{ct}
\DeclareMathOperator{\lct}{lct}
\DeclareMathOperator{\mult}{mult}
\DeclareMathOperator{\lcm}{l.c.m.}
\newcommand{\canset}{\mathcal{T}^\text{can}}
\newcommand{\lcset}{\mathcal{T}^\text{lc}}
\newcommand{\newton}{\mathcal{N}^{+}}
\newcommand{\C}{\mathbb{C}}
\newcommand{\R}{\mathbb{R}}
\newcommand{\Q}{\mathbb{Q}}
\newcommand{\Z}{\mathbb{Z}}

\begin{document}
\title[Canonical thresholds]{Smooth 3-dimensional canonical thresholds}
\dedicatory{Dedicated to the memory of my advisor \\
            Vasilii Alekseevich Iskovskikh}
\author{D.~A. Stepanov}
\address{The Department of Mathematical Modeling \\
         Bauman Moscow State Technical University \\
         2-ya Baumanskaya ul. 5, Moscow 105005, Russia}
\email{dstepanov@bmstu.ru}
\thanks{The research was supported by the Russian Grant for Scientific 
        Schools 1987.2008.1 and by the Russian Program for Development of
        Scientific Potential of the High School 2.1.1/227}
\date{}
\begin{abstract}
If $X$ is an algebraic variety with at worst canonical singularities and
$S$ is a $\Q$-Cartier hypersurface in $X$, the canonical threshold of the
pair $(X,S)$ is the supremum of $c\in\R$ such that the pair $(X,cS)$ is
canonical. We show that the set of all possible canonical thresholds of
the pairs $(X,S)$, where $X$ is a germ of smooth 3-dimensional variety,
satisfies the ascending chain condition. We also deduce a formula for
the canonical threshold of $(\C^3,S)$, where S is a Brieskorn singularity.
\end{abstract}
\maketitle

\section{Introduction}
Let $P\in X$ be a germ of a complex algebraic variety $X$ with at worst 
canonical singularities. Let $S$ be a hypersurface (not necessarily 
irreducible or reduced) in $X$ which is $\Q$-Cartier, i.~e., for some
integer $r$ the divisor $rS$ can locally be defined on $X$ by one equation.
\begin{definition}\label{D:cantr}
The \emph{canonical threshold} of the pair $(X,S)$ is
$$\ct_P(X,S)=\sup\{c\in\R\,|\,\text{the pair }(X,cS)
\text{ is canonical}\}\,.$$
\end{definition}
If we require in the above definition the variety $X$ and the pair $(X,cS)$
to be log canonical, we get the analogous notion of \emph{log canonical
threshold} $\lct_P(X,S)$ which is perhaps better known (see, e.~g., 
\cite{Singofpairs}, Sections 8, 9, 10). In the same way as it is done for
the log canonical threshold, considering an appropriate resolution of
singularities of $(X,S)$ one shows that the number $\ct_P(X,S)$ is rational
and $\ct_P(X,S)\in [0,1]\cap\Q$.
\begin{definition}
The \emph{set of $n$-dimensional canonical thresholds} is the set
$$\canset_n=\{\ct_P(X,S)\,|\,\dim X=n\}\,,$$
where $(X,S)$ varies over all pairs satisfying conditions of 
Definition~\ref{D:cantr}.
\end{definition}
We shall denote by $\lcset_n$ the corresponding \emph{set of n-dimensional 
log canonical thresholds}. The following conjecture is very important for
the Minimal Model Program. Its first part was formulated by V.~V. Shokurov 
and the second by J. Koll\'ar. 
\begin{conjecture}\label{C:lc}
\begin{description}
\item[(i)] The set $\lcset_n$ satisfies the ascending chain condition (ACC);
\item[(ii)] the set of accumulation points of $\lcset_n$ is 
$\lcset_{n-1}\setminus\{1\}$.
\end{description}
\end{conjecture}
An analog of Conjecture~\ref{C:lc} (i) for canonical thresholds is
\begin{conjecture}\label{C:can}
The set $\canset_n$ satisfies ACC.
\end{conjecture}
Conjecture~\ref{C:can} is interesting for applications to birational
geometry (\cite{Corti}). It is also a particular case of conjecture
of C. Birkar and V.~V. Shokurov about $a$-lc thresholds (\cite{BS}, 
Conjecture~1.7). We can not estimate whether the analog of 
Conjecture~\ref{C:lc} (ii) for canonical thresholds would be plausible.
\begin{remark}
The log canonical threshold $\lct_P(X,S)$ can be studied not only from
algebraic geomtry point of view. It has interpretations in terms of 
convergence of some integrals, Bernstein-Sato polynomials etc. (see 
\cite{Singofpairs}, \cite{Kollar2}). It would be interesting to find similar 
interpretations for the canonical threshold.
\end{remark}

As far as we know, all the conjectures are still open in their general form.
However, some important cases have been established. Let us denote
by $\lcset_{n,\text{smooth}}$ the \emph{set} $\{\lct_P(X,S)\,|\,\dim X=n, 
X\text{ is smooth}\}$ \emph{of smooth $n$-dimensional log canonical 
thresholds}. T. de Fernex and M. Musta\c{t}\u{a} (\cite{dFM}) and
J. Koll\'ar (\cite{Kollar2}) proved the analog of Conjecture~\ref{C:lc}, 
(ii) for the set $\lcset_{n,\text{smooth}}$. A theorem of T. de Fernex, 
L. Ein and M. Musta\c{t}\u{a} (\cite{dFM2}) states ACC for the set 
$\lcset_{n,\text{smooth}}$ for any $n$ (actually some rectricted types
of singularities of $X$ are allowed). Concerning canonical thresholds, 
there is a theorem of Yu.~G. Prokhorov (\cite{Prokhorov}, Theorem~1.4) 
describing the upper part of the set $\canset_{3}$ of 3-dimensional 
canonical thresholds.
\begin{theorem}[Prokhorov]\label{T:5/6}
If $c=\ct_P(X,S)$ for some 3-dimensional variety $X$ and $c\ne 1$,
then $c\leq 5/6$ and the bound is attained. Moreover, if $X$ is singular,
then $c\leq 4/5$ and the bound is attained.
\end{theorem}

In this paper we also restrict ourselves to a particular case of 
Conjecture~\ref{C:can}. Namely, we prove it for smooth 3-dimensional 
germs $P\in X$. Let
$$\canset_{3,\text{smooth}}=\{\ct_P(X,S)\,|\,\dim X=3, 
X\text{ is smooth}\}$$
be the set of smooth 3-dimensional canonical thresholds. Our main result
is the following.
\begin{theorem}\label{T:ACC}
The set $\canset_{3,\text{smooth}}$ satisfies the ascending chain condition.
\end{theorem}
It is not difficult to check that $\canset_{n-1,\text{smooth}}\subseteq
\canset_{n,\text{smooth}}$, so if we prove ACC in some dimension, we
automatically have it in all smaller dimensions. The proof of 
Theorem~\ref{T:ACC} is contained in Section~\ref{S:ACC}. Its main point is 
M. Kawakita's classification of 3-dimensional contractions to smooth points
(\cite{Kawakita}). In that section we assume some familiarity of the reader
with the Minimal Model Program (\cite{Matsuki}). In Section~\ref{S:Brieskorn}
we deduce a formula for the canonical threshold of a 3-dimensional Brieskorn 
singularity. It is interesting that it turns out to be much more cumbersome 
than the corresponding formula for the log canonical threshold. In 
Section~\ref{S:4/5} we slightly strengthen Theorem~\ref{T:5/6} by showing
that there are no 3-dimensional canonical thresholds between $4/5$ and 
$5/6$ (see Theorem~\ref{T:4/5}).

The author is grateful to Yu.~G. Prokhorov for attracting author's attention
to canonical thresholds and for useful suggestions during our work.
\pagebreak

\section{Ascending chain condition for smooth 3-dimensional canonical
thresholds}\label{S:ACC}

\subsection{Reduction to extremal contraction}\label{SS:reduction}
Let $P\in X$ be a germ of terminal $\Q$-factorial complex 3-dimensional 
algebraic variety and $S$ an effective integer divisor on $X$ such that the 
pair $(X,S)$ is not terminal. Let $g\colon\widetilde{X}\to X$ be an embedded 
resolution of the pair $(X,S)$. We denote by $E_i$, $i\in I$, the prime 
exceptional divisors of $g$ and by $\widetilde{S}$ the strict transform of 
$S$. Then we can write the relations
$$K_{\widetilde{X}}=g^* K_X+\sum_{i\in I} a_i E_i\,,\quad
g^*S=\widetilde{S}+\sum_{i\in I} b_i E_i$$
for some rational numbers $a_i$, $b_i$. Here $K_{\widetilde{X}}$ and $K_X$ 
stand for the canonical classes of $\widetilde{X}$ and $X$ respectively.
Now let $c\in\Q$ be the canonical threshold of the pair $(X,S)$. This
means that the pair $(X,cS)$ is canonical, i.~e., if we write
$$K_{\widetilde{X}}+c\widetilde{S}=g^*(K_X+cS)+
\sum_{i\in I}(a_i-cb_i)E_i\,,$$
then $a_i-cb_i\geq 0$ for all $i\in I$. This implies the estimate
$c\leq\frac{a_i}{b_i}$ for all $i$, and, since we assume that $c$ is the 
threshold, the equality
$$c=\ct_P(X,S)=\min\limits_{i\in I} \frac{a_i}{b_i}\,.$$
(We can always assume that the resolution $g$ has at least $1$ exceptional 
divisor with discrepancy $0$ over $(X,cS)$; this follows from 
\cite{Singofpairs}, Corollary~3.13.) This shows, in particular, that $c$ is 
indeed rational. If for a divisor $E_i$ we have $a_i-cb_i=0$, we shall say 
that $E_i$ (considered as a discrete valuation of the field $\C(X)$ of 
rational functions on $X$) \emph{realizes the canonical threshold} for the 
pair $(X,S)$. 

\begin{lemma}[cf. \cite{Complements}, Section 3]\label{L:reduction}
Let $c$ be the canonical threshold of a pair $(X,S)$ with terminal 
$\Q$-factorial 3-dimensional germ $P\in X$. Then there exists an 
extremal divisorial contraction $g'\colon X'\supset E'\to X\ni P$ such that 
its exceptional divisor $E'$ realizes the threshold $c$. (We shall also say 
that the threshold $c$ is achieved on the contraction $g'$).
\end{lemma}
\begin{proof}
We use the notation introduced before the lemma. Let us apply to 
$\widetilde{X}$ the $(K_{\widetilde{X}}+c\widetilde{S})$-Minimal Model 
Program (MMP) relative over $X$ (see \cite{Matsuki}, Ch.~11). It stops with 
a $\Q$-factorial variety $\widehat{X}$ such that the pair 
$(\widehat{X},c\widehat{S})$, where $\widehat{S}$ is the strict transform of 
$\widetilde{S}$, is terminal. Actually MMP contracts all the exceptional 
divisors of $g$ which have positive discrepancies over $(X,cS)$. Next we 
apply $K_{\widehat{X}}$-MMP over $X$ to $\widehat{X}$. It contracts all the 
exceptional divisors which remain in $\widehat{X}$ and stops with the 
variety $X$. Since $X$ was supposed to be $\Q$-factorial, the last step of 
$K_{\widehat{X}}$-MMP is an extremal divisorial contraction 
$g'\colon X'\to X$ from some $\Q$-factorial terminal variety $X'$. Let $E'$ 
be the exceptional divisor of $g'$. All the divisors that we contract on 
this stage have discrepancy $0$ over $(X,cS)$, thus $E'$ realizes the 
canonical threshold of $(X,S)$.
\end{proof}

For the rest of this section we assume $X$ to be smooth. The variety $X'$ 
obtained in Lemma~\ref{L:reduction} is $\Q$-factorial, so we again can write
\begin{equation}\label{E:discrepancy}
K_{X'}={g'}^*K_X+a'E'\,,\quad {g'}^*S=S'+b'E'\,,
\end{equation}
and since the canonical threshold $c$ is realized by $E'$, we have
$c=a'/b'$. This reduces the calculation of the canonical threshold in
the smooth 3-dimensional case to an \emph{extremal divisorial contraction},
i.~e., to a morphism with connected fibers $g'\colon X'\to X$ subject to
the following conditions: 
\begin{description}
\item[(i)] $X'$ is $\Q$-factorial with only terminal singularities;
\item[(ii)] the exceptional locus of $g'$ is a prime divisor;
\item[(iii)] $-K_{X'}$ is $g'$-ample;
\item[(iv)] the relative Picard number of $g'$ is $1$.
\end{description}

In the 3-dimensional situation $g'$ can contract the divisor $E'$ either 
onto a curve $C\subset X$ or to the point $P\in X$. In the first case
it follows from Mori's classification of smooth extremal contractions
(\cite{Mori}) that at a generic point of $C$ the morphism $g'$ is 
isomorphic to an ordinary blow up of $X$ at the curve $C$. Then it follows 
from \eqref{E:discrepancy} that $\ct_P(X,S)=1/\mult_C(f)$, where 
$\mult_C(f)$ is the multiplicity of the defining function $f$ of $S$ at a 
generic point of $C$. Thus the set $\canset_{3,\text{smooth}}$ contains the 
subset $\{1/n\,|\, n\in\mathbb{N}\}$ which satisfies ACC.

Now suppose that $g'$ contracts the divisor $E'$ to the smooth point
$P\in X$. In this case we have an important result of M. Kawakita
(\cite{Kawakita}, Theorem~1.2) classifying extremal divisorial contractions
to smooth points.
\begin{theorem}[Kawakita]\label{T:Kawakita}
Let $Y$ be a 3-dimensional $\Q$-factorial variety with 
only terminal singularities, and let $g\colon(Y\supset E)\to X\ni P$ be
an algebraic germ of an extremal divisorial contraction which contracts its 
exceptional divisor $E$ to a smooth point $P$. Then we can take local 
parameters $x$, $y$, $z$ at $P$ and coprime positive integers $a$ and $b$
such that $g$ is the weighted blow up of $X$ with its weights $(1,a,b)$.
\end{theorem}
So further we may assume that the canonical threshold of the pair $(X,S)$ 
is realized by some weighted blow up. Also, since $X$ is smooth and the
canonical threshold can be defined and calculated using analytic germs as 
well, we assume in the sequel that $P\in X$ is isomorphic to $0\in\C^n$ and 
$S$ is determined by a convergent power series $f$.

\subsection{Canonical threshold and weighted blow ups}\label{SS:wblowup}
The affine space $\mathbb{A}_{\C}^{n}$ can be given a structure of a toric 
variety $X(\tau,N)$ where $N=\Z^n$ and the cone $\tau$ is the positive octant
of the real vector space $\R^n\simeq N\bigotimes\R$. Let 
$w=(w_1,\dots,w_n)\in N\cap\tau$ be a primitive vector. The \emph{weighted 
blow up} $\sigma_w$ of the space $\mathbb{A}_{\C}^{n}\simeq\C^n$ with the 
weight vector $w$ is the toric morphism
$$\sigma_w \colon X(\Sigma_w,N)\to \C^n\simeq X(\tau,N)$$
given by the natural subdivision $\Sigma_w$ of the cone $\tau$ with a help
of the vector $w$. Certainly, a weighted blow up depends not only on its 
weights $w$ but also on the choice of a toric structure 
$\mathbb{A}_{\C}^{n}\simeq X(\tau,N)$. The variety $X(\Sigma_w,N)$ is
$\Q$-factorial and can be covered by $n$ affine charts. The $i$th chart
is isomorphic to $\C^n/\Z_{w_i}$ where the cyclic group acts with weights 
$(-w_1,\dots,-w_{i-1},1,-w_{i+1},\dots,-w_n)$, and the morphism $\sigma_w$
is given in this chart by the formulae
$$x_i=y_{i}^{w_i}\,,\quad x_j=y_j y_{i}^{w_j}\,, \; j\ne i$$
where $x_1,\dots,x_n$ are the coordinates on the target and $y_1,\dots,y_n$
on the source space (see, e.~g., \cite{Complements}, 3.7).

Given a hypersurface $S=\{f=0\}$ in $\C^n$, we can estimate the canonical
threshold $\ct_0(\C_n,S)$ with a help of weighted blow ups. Namely, for
any weight vector $w$ (not equal to $0$ or to a vector $e_i$ of the 
standard basis) we have
$$\ct_0(\C^n,S)\leq \frac{w_1+\dots+w_n-1}{w(f)}\,$$
where $w(f)$ is the least weight of a monomial appearing in $f$ with 
respect to the weights $w_1,\dots,w_n$ (see \cite{Prokhorov}). Moreover,
if we know that the canonical threshold of $(\C^n,S)$ is achieved on some
weighted blow up, we can calculate it as
\begin{equation}\label{E:ct}
\ct_0(\C^n,S)=
\min_{w\ne 0,e_i,i=1,\dots,n} \frac{w_1+\dots+w_n-1}{w(f)}\,.
\end{equation}
where the minimum is taken over all integer vectors in $\tau$. It is
possible that the denominator in \eqref{E:ct} is 0 for some weights $w$;
such fractions should be treated as $+\infty$.

Recall that the \emph{extended Newton diagram} $\Gamma^+(f)$ of a 
polynomial (or a series) $f=\sum_m a_m x^m$ is the convex hull in $\R^n$ 
of the set $\{m+\mathbb{R}_{\geq 0}^{n}\,|\,a_m\ne 0\}$. Note that the
denominator $w(f)$ in the above formulae depends only on $\Gamma^+(f)$
but not on $f$ itself. Let us introduce a new set
$\canset_{n,\text{smooth,w}}=$ the set of smooth $n$-dimensional canonical 
thresholds which can be realized by some weighted blow up.
\begin{remark}
The set $\canset_{n,\text{smooth,w}}$ may look a bit artificial. However, 
it contains, for example, the set of smooth $3$-dimensional canonical 
thresholds (for $n=3$; see subsection~\ref{SS:reduction}), or the set of 
canonical thresholds achieved on hypersurfaces $S$ defined by series $f$ 
nondegenerate with respect to their Newton diagrams.
\end{remark}
Now Theorem~\ref{T:ACC} follows from the next result.
\begin{theorem}\label{T:wACC}
The set $\canset_{n,\text{smooth,w}}$ satisfies ACC.
\end{theorem}
The \emph{proof} of Theorem~\ref{T:wACC} will follow from the 2 lemmas
below. 

Let us denote by $\newton_n$ the set of all possible extended 
Newton diagrams in $\R^n$. Clearly this set is ordered with respect to 
the inclusion relation $\subseteq$. The next lemma is perhaps well known
for specialists on singularity theory or on Gr\"obner bases. But since we
do not know a good reference, we state it with a proof.
\begin{lemma}\label{L:ACC}
Any infinite sequence of elements of the ordered set $\newton_n$ contains
a monotonous non increasing subsequence. In particular, the set $\newton_n$ 
satisfies ACC.
\end{lemma}
\begin{proof} It is more convinient to use the ideals of semigroup
$(\Z_{\geq 0}^{n},+)$ (see \cite{Kh&Ch}, \S\,1.1) in the proof instead 
of extended Newton diagrams. With any diagram $\Gamma^+$ we can associate
the ideal generated by its vertices. From any ideal we can construct an
extended Newton diagram by taking its convex hull in $\R_{\geq 0}^{n}$. 
Despite this correspondence is not one-to-one, it should be clear that
it suffices to prove the lemma for ideals. 

The proof goes by induction on the dimension $n$. For $n=1$ the statement
is obvious. Let $\Gamma_{1}^{+}$, $\Gamma_{2}^{+},\dots$ be a sequence
of ideals of the semigroup $\Z_{\geq 0}^{n}$. We can project
every ideal $\Gamma_{i}^{+}$ to the coordinate axes of the space $\R^n$. 
Every projection gives an ideal in the semigroup $\Z_{\geq 0}$.
First let us suppose that for some axis, say, the last one, the so obtained
sequence of projections is unbounded, i.~e., for any $M>0$ there is a
number $k$ such that the projection of the ideal $\Gamma_{k}^{+}$ to the
last axis lies above $M$. Choosing a subsequence we can assume that
in fact the projection of any ideal $\Gamma_{i}^{+}$ lies above $M$ for
$i\geqslant k$. The projection of every ideal $\Gamma_{i}^{+}$ to the
coordinate hyperplane containing the first $n-1$ coordinate axes
is an ideal in the semigroup $\Z_{\geq 0}^{n-1}$. By the induction
hypothesis we can choose a subsequence such that the corresponding sequence
of projections to the given hyperplane is monotonous non increasing.
To simpify the notation we assume that already the sequence $\Gamma_{1}^{+}$, 
$\Gamma_{2}^{+},\dots$ has this property. Consider the ideal 
$\Gamma_{1}^{+}$. Let us denote by $M$ the maximal of the last coordinates 
of its vertices (we use here the fact that any ideal of 
$\Z_{\geq 0}^{n}$ is a \emph{finite} union of the shifted coordinate
octants; this is called the \emph{first finiteness property} of the 
semigroup $\Z_{\geq 0}^{n}$, see \cite{Kh&Ch}, \S\,1.1, Theorem~1). 
Starting from some number $k$ all the ideals $\Gamma_{i}^{+}$ lie above
$M$ with respect to the last coordinate, and thus they all are contained
in $\Gamma_{1}^{+}$. Going on in the same way we construct the needed
non increasing subsequence of ideals.

Now suppose that the projections of the sequence $\Gamma_{1}^{+}$, 
$\Gamma_{2}^{+},\dots$ to all the coordinate axes are bounded. This
means that all the ideals $\Gamma_{i}^{+}$ have a vertex inside some
fixed polytope. But every bounded polytope has only finite number of
integer points, thus, again choosing a subsequence if necessary,
we can assume that all $\Gamma_{i}^{+}$ have a common vertex $m$. It
follows that all the ideals $\Gamma_{i}^{+}$ contain also a common
shifted coordinate octant $m+\R_{\geq 0}^{n}$. Consider the complement
to this octant in $\Z_{\geq 0}^{n}$. It is clear that it is a union
of finite number of shifted coordinate subsemigroups of the semigroup
$\Z_{\geq 0}^{n}$ (this is a particular case of the \emph{second
finiteness property} of the semigroup $\Z_{\geq 0}^{n}$, 
\cite{Kh&Ch}, \S\,1.1, Theorem~2). The dimension of such a semigroup
is not greater than $n-1$, and the intersection of any ideal
$\Gamma_{i}^{+}$ with such a semigroup is an ideal of it. Applying the
induction hypothesis and repeatedly (but finite number of times!) choosing 
a subsequence, we conclude that outside the octant $m+\R_{\geq 0}^{n}$ 
our sequence $\Gamma_{1}^{+}$, $\Gamma_{2}^{+},\dots$ can be assumed to be
non increasing. But taking the union with the same octant obviously does not
change this property. This finishes the proof.
\end{proof}

The formula \eqref{E:ct} formally defines the canonical threshold of
an extended Newton diagram $\Gamma^+$ which we shall denote by 
$\ct(\Gamma^+)$. We can consider it as a map
$$\ct\colon \newton_n\to \R\,.$$
\begin{lemma}\label{L:monoton}
The map $\ct$ is monotonous, i.~e.,
$$\Gamma_{1}^{+}\subseteq\Gamma_{2}^{+} \Rightarrow 
\ct(\Gamma_{1}^{+})\leq\ct(\Gamma_{2}^{+})\,.$$
\end{lemma}
\begin{proof}
Let $w$ be a vector in $\Z_{\geq 0}^{n}$, $w\ne 0,e_1,\dots,e_n$. It defines 
a rational function $w\colon\R^{n} -\to \R$
$$w(x)=\frac{w_1+\dots+w_n-1}{w_1 x_1+\dots+w_n x_n}\,.$$
The level set $w=c$ of this function ($w$ is fixed, $x$ varies) is a
hyperplane with a non negative normal vector $w$. The smaller $c>0$ we take,
the further from the origin the level set $w=c$ is.

Suppose that the canonical threshold $c=\ct(\Gamma_{2}^{+})$ of the diagram 
$\Gamma_{2}^{+}$ is realized by the weight vector $w$ and this threshold is 
attained on a vertex $m$ of the diagram $\Gamma_{2}^{+}$, $c=w(m)$. The 
minimum $c'$ of the function $w$ on the diagram $\Gamma_{1}^{+}$ is not 
greater than $c$ because $\Gamma_{1}^{+}$ is situated ``above'' the diagram 
$\Gamma_{2}^{+}$. The threshold $\ct(\Gamma_{1}^{+})$ can only be less or
equal to $c'$.
\end{proof}

To finish the proof of Theorem~\ref{T:wACC}, suppose that there exists a 
strictly increasing sequence $c_1<c_2<\dots$ of canonical thresholds from
$\canset_{n,\text{smooth,w}}$. Let $\Gamma_{1}^{+}$, $\Gamma_{2}^{+},\dots$
be a sequence of extended Newton diagrams such that $\ct(\Gamma_{k}^{+})=
c_k$. We can not have an inclusion $\Gamma_{i}^{+}\supseteq\Gamma_{j}^{+}$
for any $i<j$ because of Lemma~\ref{L:monoton}. But this contradicts 
Lemma~\ref{L:ACC}.

\section{Canonical threshold for Brieskorn singularities in $\C^3$}
\label{S:Brieskorn}
A \emph{Brieskorn singularity} is a hypersurface singularity $S$ is $\C^n$
given by the equation
$$x_{1}^{a_1}+x_{2}^{a_2}+\dots+x_{n}^{a_n}=0\,.$$
For $n=3$ we shall assume that $S$ is given by
\begin{equation}\label{E:Brieskorn}
x^a+y^b+z^c=0\,,
\end{equation}
where $2\leq a\leq b\leq c$. The log canonical threshold of the pair
$(\C^n,S)$ can be determined by the formula (\cite{Singofpairs}, 8.15)
$$\lct_0(\C^n,S)=\min\{\frac{1}{a_1}+\frac{1}{a_2}+\dots+
\frac{1}{a_n},1\}\,.$$
In this section we calculate the 3-dimensional canonical threshold 
$\ct_0(\C^3,S)$. Brieskorn singularities are nondegenerate with respect
to their Newton diagrams, thus they admit embedded toric resolutions
(for the definition of nondegeneracy and construction of embedded toric
resolution see \cite{Varchenko}). It follows, in particular, that their
canonical thresholds are realized by weighted blow ups and we can apply
formula~\eqref{E:ct} from subsection~\ref{SS:wblowup}. In the case of
3-dimensional Brieskorn singularities it takes the form
\begin{equation}\label{E:ctBrieskorn}
\ct_0(\C^3,S)=
\min_{w\ne e_1,e_2,e_3,0} \frac{w_1+w_2+w_3-1}{\min\{aw_1,bw_2,cw_3\}}\,,
\end{equation}
where the minimum is taken over all vectors $w$ from $\Z_{\geq 0}^{3}$.
\begin{remark}\label{R:Kawakita}
Formula~\eqref{E:ctBrieskorn} is not a direct consequence of 
Theorem~\ref{T:Kawakita} and subsection~\ref{SS:reduction}. Indeed, the 
theorem states that there exist \emph{some} coordinates in which the 
extremal contraction realizing the canonical threshold is a weighted 
blow up. But in those coordinates the equation of our singularity must not 
be of Brieskorn type or even nondegenerate.
\end{remark}
\begin{lemma} \label{L:order}
Let $S\subset\C^3$ be a Brieskorn singularity given by 
equation~\eqref{E:Brieskorn}. Suppose that a weight vector $w=(p,q,r)$ 
realizes the minimum in \eqref{E:ctBrieskorn}. Then $p\geq q\geq r$.
\end{lemma}
\begin{proof}
It is clear that the vector realizing the canonical threshold satisfies
$p,q,r\ne 0$. Assume, for example, that $p<q$. Then
$$\ct_0(\C^3,S)=\frac{p+q+r-1}{\min\{ap,cr\}}$$
and $q\geq 2$. But in this case we could take $w'=(p,q-1,r)$ 
instead of $w$ and $w'$ would give strictly smaller canonical threshold, 
a contradiction. Other inequalities can be considered similarly.
\end{proof}

\begin{lemma}\label{L:r1}
A weight vector $w$ giving the canonical threshold of a Brieskorn 
singularity~\eqref{E:Brieskorn} can always be chosen in the from $w=(p,q,1)$, 
where $p$ and $q$ are coprime positive integers.
\end{lemma}
\begin{proof}
Consider a piecewise rational function $h$ determined on $\R_{>0}^{3}$ by
the formula
$$h(w)=\frac{w_1+w_2+w_3-1}{\min\{aw_1,bw_2,cw_3\}}\,.$$
Its level set $h(w)=s$, $s\geq 0$, coincides with the lateral surface of
a tetrahedron $\Delta_s$ with vertices $(1,0,0)$, $(0,1,0)$, $(0,0,1)$ 
(forming the base face of the tetrahedron) and with the last vertex 
$$\frac{1}{1/a+1/b+1/c-s}
\left(\frac{1}{a},\frac{1}{b},\frac{1}{c}\right)$$
on the line $aw_1=bw_2=cw_3$. If $s<s'$, then $\Delta_s\subset\Delta_{s'}$.
We see that the canonical threshold of a Brieskorn 
singularity~\eqref{E:Brieskorn} can be found with a help of the following
process. For every $s\geq0$ we construct the tetrahedron $\Delta_s$ and find
the minimal $s_0$ for which $\Delta_{s_0}$ contains an integer point 
$(p,q,r)$ with $p,q,r>0$ on its lateral border. Then $\ct_0(\C^3,S)=s_0$ 
and the threshold is realized by the weight vector $(p,q,r)$.

Now suppose that $\ct_0(\C^3,S)=s_0=h(p,q,r)$ and $r\geq 2$. From 
Lemma~\ref{L:order} we know that $p\geq q\geq r$. Consider also a
tetrahedron $\Delta_w$ with vertices $(1,0,0)$, $(0,1,0)$, $(0,0,1)$, and
$(p,q,r)$. Obviously $\Delta_w\subset\Delta_{s_0}$ and we prove our lemma
if we show that the intersection of $\Delta_w$ with the plane $w_3=1$
contains a positive integer point. Indeed, if $(p',q',1)$ is such a point
and $d$ is the greatest common divisor of $p'$ and $q'$, the point
$(p'/d,q'/d,1)$ also lies in $\Delta_w$ and gives smaller value of $h$.
The intersection $\Delta_w\cap\{w_3=1\}$ is a triangle with
vertices $(0,0)$, $1/r(p+r-1,q)$, $1/r(p,q+r-1)$ (the third coordinate
$w_3=1$ is omitted here). It is a bit more convenient to multiply
everything by $r$ and to show that the triangle $OPQ$, $O=(0,0)$, 
$P=(p,q+r-1)$, $Q=(p+r-1,q)$, contains a positive integer point with
coordinates $0\mod r$ (see Figure~\ref{F:intpoints}).
\begin{figure}[hbt]
\begin{picture}(7,6)
\put(1,1){\vector(1,0){5.5}} \put(6.5,0.7){$w_1$}
\put(1,1){\vector(0,1){4.5}} \put(0.5,5.5){$w_2$}
\put(0.7,0.7){$O$}
\put(1,1){\line(5,2){5}} \put(6,3){\line(-1,1){1}} \put(1,1){\line(4,3){4}}
\put(6,2.7){$Q(p+r-1,q)$} \put(4.3,4.2){$P(p,q+r-1)$}
\put(5.5,3.5){\line(-1,0){1.165}} \put(4.75,2.5){\line(-1,0){1.74}}
\put(6,3){\line(-1,0){2.33}}
\put(2.6,2.45){$A$} \put(3.3,3){$E$} \put(4,3.5){$B$} 
\put(5.6,3.5){$C$} \put(4.7,2.2){$D$}
\end{picture}
\caption{Integer points in the triangle}\label{F:intpoints}
\end{figure}
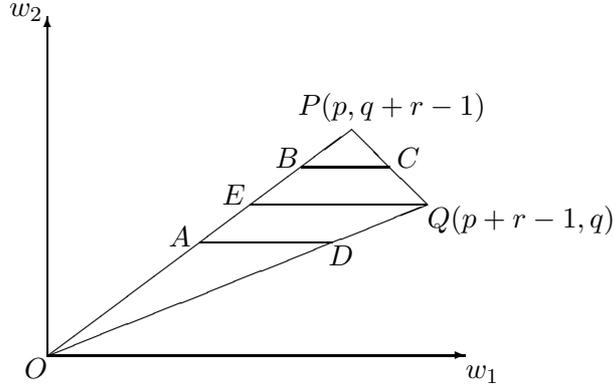
Other points in Figure~\ref{F:intpoints} have the following meaning.
$E$ is the intersection point of the lines $OP$ and $w_2=q$. Its coordinates
are $(pq/(q+r-1),q)$. The segment $BC$ is the middle line of the triangle 
$EPQ$ and the segment $AD$ lies on the line $w_2=q-r/2$. 

Note that the point $Q$ lies under the diagonal $w_1=w_2$. Thus if the 
point $P$ lies above the diagonal, then the triangle $OPQ$ contains already 
the point $(r,r)$. So let us assume $p>q+r-1$. Choose an integer $k\geq 2$ 
such that $(k-1)(q+r-1)< p\leq k(q+r-1)$. Next we consider 2 cases:
$q\leq 2r-1$ and $q\geq 2r$. Suppose first that $q\leq 2r-1$. Consider a
transformed triangle $OP'Q'$ obtained from $OPQ$ with a help of the
unimodular transformation
$$
\begin{pmatrix}
1 & -(k-1) \\
0 & 1
\end{pmatrix}\,.
$$
The point $P'(p-(k-1)(q+r-1),q+r-1)$ lies above the diagonal. Indeed,
$p-(k-1)(q+r-1)\leq q+r-1$. On the other hand, the point 
$Q'(p-(k-1)q+r-1,q)$ lies under the diagonal: 
$$p-(k-1)q+r-1>(k-1)(q+r-1)-(k-1)q+r-1=$$
$$=2r-1+(k-2)r-(k-1)\geq q+(k-2)r-(k-1)\geq q-1$$
because $r$, $k\geq 2$. It follows that the triangle $OP'Q'$ contains the 
point $(r,r)$. But then the triangle $OPQ$ also has a positive integer
point $0\mod r$.

Now suppose that $q\geq 2r$. Note that the length of $EQ$ is 
$$p+r-1-\frac{pq}{q+r-1}=(r-1)\frac{p+q+r-1}{q+r-1} > 2(r-1)$$
and hence $BC>r-1$. $AD\geq (3/4)EQ > (3/2)(r-1)$. The last number is
greater than $r$ for $r\geq 3$. If we have a segment $[x,y]$ with integer 
$x$ or $y$ on the real line $\R$, and if this segment has length 
$\geq r-1$, then it necessarily contains a point $0\mod r$ (perhaps as one 
of its border points). If the length of the segment is $\geq r$, it
always contains a point $0\mod r$ no matter whether $x$ or $y$ are 
integers. From this observations it easily follows that for $r\geq 3$ 
already the pentagon $ABCQD$ contains a point $0\mod r$. We leave the 
details to the reader. 

So it remains to consider the case when $r=2$ and $q\geq 4$. 
Moreover, we can assume that $p$ and $q$ are odd because otherwise already
one of the points $P$, $Q$, or $(p,q)$ will have coordinates $0$ modulo
$r$. Let us check if the point $(p-1,q-1)$ lies in the triangle
$OPQ$. This holds if
$$\frac{q-1}{p-1}\geq\frac{q}{p+1}\,,$$
or $p\leq 2q-1$. On the other hand,
$$AD=\frac{q-1}{q}\cdot\frac{p+q+1}{q+1}\geq 2$$
holds if 
$$p\geq \frac{(q+1)^2}{q-1}>q+1\,.$$
Two inequalities $p\leq 2q-1$ and $p>q+1$ cover all possibilities for
$p$ and $q$. Thus the triangle $OPQ$ always contains the desired point.
\end{proof}
\begin{remark}
Again Lemma~\ref{L:r1} is not a direct consequence of Kawakita's 
Theorem~\ref{T:Kawakita}, see Remark~\ref{R:Kawakita}
\end{remark}

Now we are ready to deduce a formula for the canonical threshold of a
Brieskorn singularity. To do this, let us introduce some new notation.
Denote by $L$ the point $(c/a,c/b,1)$ of the intersection of the line 
$aw_1=bw_2=cw_3$ with the plane $w_3=1$. Fix a real number $s$ and consider 
the intersection of the tetrahedron $\Delta_s$ with the plane $w_3=1$
(see the proof of Lemma~\ref{L:r1}). For $s<1/a+1/b$ this intersection is 
empty; for $s=1/a+1/b$ it is the segment $OL$; for $s>1/a+1/b$ it is a 
triangle $OMN$ where $M=(c/a,sc-c/a,1)$, $N=(sc-c/b,c/b,1)$ (see 
Figure~\ref{F:intpoints2}). 
\begin{figure}[hbt]
\begin{picture}(6,5)
\put(0.5,0.5){\vector(1,0){5}} \put(5.5,0.2){$w_1$}
\put(0.5,0.5){\vector(0,1){4}} \put(0,4.5){$w_2$}
\put(0.2,0.2){$O$}
\put(0.5,0.5){\line(2,1){4}} \put(4.5,2.5){\line(-1,1){1}} 
\put(0.5,0.5){\line(1,1){3}}
\put(0.5,0.5){\line(3,2){3}} \put(3.5,2.5){\line(1,0){1}}
\put(3.5,2.5){\line(0,1){1}}
\put(3.5,2.5){\circle*{0.1}}
\put(3.1,2.5){$L$} \put(3.4,3.6){$M$} \put(4.6,2.3){$N$}
\put(0.5,1.5){\line(1,0){4.5}} \put(0.2,1.4){$k$} 
\put(2.2,1.5){\circle*{0.1}}
\end{picture}
\caption{Minimizing the function $h$}\label{F:intpoints2}
\end{figure}
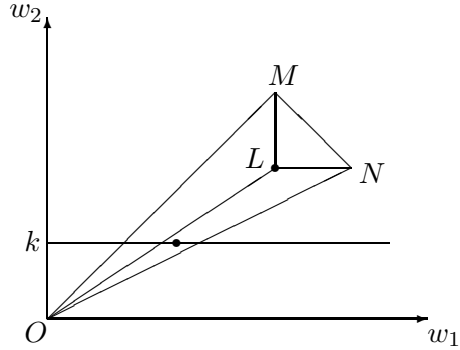
This almost immediately implies
\begin{lemma}
Let $S$ be a Brieskorn singularity of the form \eqref{E:Brieskorn}.
Then $\ct_0(\C^3,S)\geq 1/a+1/b$. Moreover, if $c\geq \lcm(a,b)$, 
where $\lcm$ is the least common multiple, then $\ct_0(\C^3,S)=1/a+1/b$.
\end{lemma}
\begin{proof}
Only the ``moreover'' part of the lemma needs a proof. If $c\geq m=
\lcm(a,b)$, then the segment $OL$ contains the point $(m/a,m/b,1)$.
$h(m/a,m/b,1)=1/a+1/b$, but we know from the first part of the lemma
that $1/a+1/b$ is the least possible value of the canonical threshold.
The proof is completed.
\end{proof}

When $c<\lcm(a,b)$, we have a kind of integer programming problem:
minimize the piecewise rational function $h$ on positive integer points
of the plane $w_3=1$. Equivalently, we have to determine the minimal $s$
such that the triangle $OMN$ contains a positive integer point. The point
$w$ on which $h$ attains its minimum is on the following list:
\begin{description}
\item[(i)] $w$ is the nearest to $OL$ integer point lying on the horizontal
line $w_2=k$, $1\leq k\leq \lfloor c/b\rfloor$ to the right from
the segment $OL$ (see Figure~\ref{F:intpoints2}); in this case 
$w=(\lceil kb/a\rceil,k,1)$; here $\lfloor\cdot\rfloor$ denotes the lower 
and $\lceil\cdot\rceil$ the upper integer;
\item[(ii)] $w$ is the nearest to $OL$ integer point lying on the vertical
line $w_1=k$, $1\leq k\leq \lfloor c/a\rfloor$ above the segment 
$OL$; in this case $w=(k,\lceil ka/b\rceil,1)$;
\item[(iii)] $w$ is ``the first'' integer point in the triangle $LMN$;
then $w=(\lceil c/a\rceil,\lceil c/b\rceil,1)$.
\end{description}
Let
$$s_1=\min\limits_{1\leq k\leq \lfloor c/b\rfloor} 
\{h(\lceil kb/a\rceil,k,1)\}=\min\limits_{1\leq k\leq \lfloor c/b\rfloor} 
\left\{\frac{1}{b}+
\frac{1}{kb}\left\lceil\frac{kb}{a}\right\rceil\right\}\,,$$
$$s_2=\min\limits_{1\leq k\leq \lfloor c/a\rfloor} 
\{h(k,\lceil ka/b\rceil,k,1)\}=\min\limits_{1\leq k\leq \lfloor c/a\rfloor} 
\left\{\frac{1}{a}+
\frac{1}{ka}\left\lceil\frac{ka}{b}\right\rceil\right\}\,,$$
$$s_3=h(\lceil c/a\rceil,\lceil c/b\rceil,1)=
\frac{\lceil c/a\rceil+\lceil c/b\rceil}{c}\,.$$
We summarize what we did in this section in the following result.
\begin{theorem}
Let $S$ be a Brieskorn singularity~\eqref{E:Brieskorn}. If 
$\lcm(a,b)\leq c$, then $\ct_0(\C^3,S)=1/a+1/b$; otherwise $\ct_0(\C^3,S)=
\min\{s_1,s_2,s_3,1\}$ (notation as above).
\end{theorem}

\begin{example}
Consider a Brieskorn singularity
$$x^3+y^7+z^{11}=0\,.$$
Using our formulae we get
$$s_1=\min\limits_{1\leq k\leq 1}\{1/7+(1/7)\cdot 3\}=\frac{4}{7}\,,$$
$$s_2=\min\limits_{1\leq k\leq 3}\{1/3+(1/3k)\lceil(3k)/7\rceil\}=$$
$$=\min\{1/3+1/3,1/3+1/6,1/3+2/9\}=\frac{1}{2}\,,$$
$$s_3=(4+2)/11=\frac{6}{11}\,.$$
It follows that $\ct_0(\C^3,S)=1/2=s_2$ and it is achieved on
the weighted blow up with weights $(2,1,1)$.
\end{example}

\begin{example}
Let $S$ be a Brieskorn singularity
$$x^5+y^6+z^{29}=0\,.$$
We have
$$s_1=\min\limits_{1\leq k\leq 4}\{1/6+1/(6k)\lceil(6k)/5\rceil\}=$$
$$=\min\{1/6+1/3,1/6+1/4,1/6+2/9,1/6+5/24\}=\frac{3}{8}\,,$$
$$s_2=\min\limits_{1\leq k\leq 5}\{1/5+1/(5k)\lceil(5k)/6\rceil\}=
\frac{2}{5}\,,$$
$$s_3=(6+5)/29=\frac{11}{29}\,.$$
It follows that $\ct_0(\C^3,S)=3/8=s_1$ and it is achieved on the weighted
blow up $(5,4,1)$.
\end{example}

\begin{example}
Now let $S$ be a Brieskorn singularity
$$x^{12}+y^{18}+z^{35}=0\,.$$
We get
$$s_1=\min\limits_{1\leq k\leq 1}\{1/18+1/9\}=\frac{1}{6}\,,$$
$$s_2=\min\limits_{1\leq k\leq 2}\{1/12+1/(12k)\lceil(2k)/3\rceil\}=
\frac{1}{6}\,,$$
$$s_3=(3+2)/35=\frac{1}{7}\,.$$
It follows that $\ct_0(\C^3,S)=1/7=s_3$ and it is achieved on the weighted
blow up $(3,2,1)$.
\end{example}

\section{The upper part of the canonical set}\label{S:4/5}
In this section we strengthen Theorem~\ref{T:5/6} of Yu.~G. Prokhorov
describing the upper part of the set $\canset_{3}$ of 3-dimensional 
canonical thresholds.
\begin{theorem}\label{T:4/5}
The intersection $\canset_3\cap[4/5,1]$ is precisely $\{4/5,5/6,1\}$.
\end{theorem}
\begin{proof}
Recall that if $X$ is singular, then $\ct_P(X,S)\leq 4/5$ 
(Theorem~\ref{T:5/6}), and if $S$ has non isolated singularities in a
neighborhood of $P$, then $\ct_P(X,S)\leq 1/2$ 
(subsection~\ref{SS:reduction}). Thus we may assume that $S$ is a
hypersurface in $\C^3$ with isolated singularity at the origin. Then
by Kawakita's Theorem~\ref{T:Kawakita} and subsection~\ref{SS:reduction}
the canonical threshold $\ct_0(\C^3,S)$ is achieved on some weighted
blow up.

Let $S$ be given in $\C^3$ by an equation $f=0$. We shall analyze the 
Newton diagram $\Gamma(f)$ of $f$ and show that the canonical threshold of 
$S$ can be $1$, $5/6$, $3/4$ or smaller. First note that if the Newton 
diagram of $f$ lies above the plane $\alpha+\beta+\gamma=3$, then 
$\ct_0(\C^3,S)\leq 2/3$. Indeed, in this case $\Gamma^+(f)$ is contained
in the extended Newton diagram of the singularity 
$$x^3+y^3+z^3=0$$
which has canonical threshold $2/3$. Thus
by Lemma~\ref{L:monoton} $\ct_0(\C^3,S)\leq 2/3$.

We see that the function $f$ necessarily has monomials of degree $2$.
If the second differential of $f$ has rank $2$ or $3$, $f$ is isomorphic
to a Du Val singularity of type $A_n$. In this case its canonical 
threshold is $1$. But the same holds even if the second differential of
$f$ has rank $1$ and $f$ has at least $2$ monomials of degree $2$.
Indeed, recall that the canonical threshold of $S$ depends only on
the Newton diagram $\Gamma(f)$. But then we can perturb the coefficients
of $f$ in such a way that the second differential becomes of rank $\geq 2$.
It follows that we can assume that $f$ has the form
$$f=x^2+\text{ terms of degree}\geq 3\,.$$
Moreover, making a substitution $x'=x\sqrt{1+\dots}$ (it does not violate
the property that the canonical threshold of $f$ is achieved on a weighted
blow up) we can assume that $x^2$ is the only monomial of $f$ containing 
$x$ with degree $\geq 2$.

Further, let us compare $f$ with the Brieskorn singularity
$$x^2+y^4+z^4=0\,.$$
Its canonical threshold is $3/4$ (see Section~\ref{S:Brieskorn}). It 
follows that if $\Gamma(f)$ lies above the plane $2\alpha+\beta+\gamma=4$, 
then $\ct_0(\C^3,S)\leq 3/4$. Therefore we can suppose that $f$ has a 
monomial of degree $3$. If this monomial is $y^2z$, we again conclude that 
$f$ is a Du Val singularity (this time of type $D_n$) and $\ct_0(\C^3,S)=1$.
Thus we assume
$$f=x^2+y^3+\text{ other terms of degree}\geq 3\,.$$

Next we compare $f$ with the Brieskorn singularity
$$x^2+y^3+z^6=0\,.$$
Its canonical threshold is $5/6$. Suppose that $f$ contains monomials
$x^\alpha y^\beta z^\gamma$ lying below the plane 
$3\alpha+2\beta+\gamma=6$. Possible monomials are $z^3$, $z^4$, $z^5$, 
$yz^2$, $yz^3$, $xz^2$. In this case $f$ is a Du Val singularity of
type $D_n$ or $E_n$ and its canonical threshold is $1$. Therefore
it remains to consider the case when $\Gamma(f)$ lies above the plane
$3\alpha+2\beta+\gamma=6$ and $\ct_0(\C^3,S)\leq 5/6$. Let us show that
in fact $\ct_0(\C^3,S)=5/6$.

Since we suppose that $f$ defines an isolated singularity, it has
monomials of the form $z^n$, $xz^n$, or $yz^n$. Hence we can estimate
the canonical threshold of $S$ from below comparing it with the
nondegenerate singularities
$$x^2+y^3+z^n=0\,,\; n\geq 6\,,$$
$$x^2+y^3+xz^n=0\,,\; n\geq 3\,,$$
or
$$x^2+y^3+yz^n=0\,,\; n\geq 4\,.$$ 
The first and the second singularity are easily seen to be isomorphic
to Brieskorn singularities with canonical threshold $5/6$. Let us prove
by a direct computation that the canonical threshold of the third
singularity $S'\subset\C^3$ is also $5/6$. 

Consider the weighted blow up $\sigma$ of $\C^3$ with weights $(3,2,1)$. 
The blown up variety $\widetilde{\C^3}$ is covered by $3$ affine charts
(see subsection~\ref{SS:wblowup}). In the first isomorphic to 
$\C^3/\Z_3(1,1,2)$ the strict transform $S'_{\widetilde{\C^3}}$ of $S'$ is
isomorphic to
$$1+y^3+x^{n-4}yz^n=0$$
and is nonsingular. In the second chart $\C^3/\Z_2(1,1,1)$ the strict
transform
$$x^2+1+y^{n-4}z^n=0$$
is again nonsingular. Note that the quotient singularities of the first 2
charts are terminal and $S'_{\widetilde{\C^3}}$ does not pass through
them. The third chart is isomorphic to $\C^3$ and the strict transform to
$$x^2+y^3+yz^{n-4}=0\,,$$
i.~e., to a singularity of the same form but with smaller $n$. Computing
discrepancy we get
$$K_{\widetilde{\C^3}}+(5/6)S'_{\widetilde{\C^3}}=
\sigma^*(K_{\C^3}+(5/6)S')\,.$$
Therefore by induction we show that $\ct_0(\C^3,S')=5/6$ and finish
the proof of our theorem.
\end{proof}

\end{document}